\RequirePackage{plautopatch}
\documentclass[12pt]{article}
\usepackage[top=25truemm,bottom=20truemm,left=20truemm,right=20truemm]{geometry}
\usepackage[noadjust]{cite}
\usepackage{hyperref}
\usepackage{xcolor}
\usepackage{amsthm} 
\usepackage{amsmath,amssymb,amsfonts,amscd} 
\usepackage[all]{xy} 
\usepackage{mathtools}
\usepackage{cleveref} 
\usepackage{graphicx}

\newcommand{\address}[1]{\bigskip{\small\noindent #1 \par}}
\newcommand{\email}[1]{{\small\noindent\textit{Email address}: \texttt{#1} \par}}

\theoremstyle{plain}
\newtheorem{thm}{Theorem}[section]
\crefname{thm}{Theorem}{Theorem}
\newtheorem{cor}[thm]{Corollary}
\crefname{cor}{Corollary}{Corollary}
\newtheorem{lem}[thm]{Lemma}
\crefname{lem}{Lemma}{Lemma}
\newtheorem{prop}[thm]{Proposition}
\crefname{prop}{Proposition}{Proposition}
\theoremstyle{definition}
\newtheorem{dfn}[thm]{Definition}
\crefname{dfn}{Definition}{Definition}
\newtheorem{ex}[thm]{Example}
\crefname{ex}{Example}{Example}
\newtheorem{rmk}[thm]{Remark}
\crefname{rmk}{Remark}{Remark}
\crefname{figure}{Figure}{Figures}
\crefname{section}{Section}{Sections}

\newcommand{\Map}{\mathrm{Map}}

\newcommand{\RR}{\mathbb{R}}
\newcommand{\ZZ}{\mathbb{Z}}

\newcommand{\Top}{\mathrm{Top}}
\newcommand{\Grp}{\mathrm{Grp}}
\newcommand{\Aut}[1]{\mathrm{Aut}(#1)}
\newcommand{\Inn}[1]{\mathrm{Inn}(#1)}
\newcommand{\Fix}[1]{\mathrm{Fix}(#1)}

\newcommand{\Qdle}{\mathrm{Qdle}}
\newcommand{\Dis}[1]{\mathrm{Dis}(#1)}
\newcommand{\GAlex}{\mathrm{GAlex}}
\newcommand{\Core}{\mathrm{Core}}
\newcommand{\freeunion}{\sqcup^{\mathrm{free}}}

\title{On the Euler characteristics for quandles}
\author{Ryoya Kai and Hiroshi Tamaru}
\date{}

\begin{document}
\maketitle
\begin{abstract}
  A quandle is an algebraic system
  whose axioms generalize
  the algebraic structure of 
  the point symmetries of symmetric spaces.
  In this paper, 
  we give a definition of 
  Euler characteristics for quandles. 
  In particular,
  the quandle Euler characteristic 
  of a compact connected Riemannian 
  symmetric space coincides with
  the topological Euler characteristic.
  Additionally,
  we calculate the Euler characteristics 
  of some finite quandles,
  including generalized Alexander quandles,
  core quandles,
  discrete spheres, and discrete tori.
  Furthermore, we prove several properties of 
  quandle Euler characteristics,
  which suggest that they share similar properties with 
  topological Euler characteristics.
\end{abstract}
\section{Introduction}

The Euler characteristic for a topological space 
is one of the classical and important topological invariants.
In general, 
the Euler characteristics
can be defined for
topological spaces using Betti numbers or cell complexes.
Hopf and Samelson \cite{Hopf-1941-SatzüberWirkungsräumeGeschlossener}
developed a method for calculating the Euler characteristics 
of homogeneous spaces of compact Lie groups.
A \emph{symmetric space} is a space with a point symmetry at each point.
A compact connected symmetric space is a homogeneous space
under the action of the compact Lie group generated by these point symmetries.
In conclusion, 
the Euler characteristic of a compact connected Riemannian symmetric space
can be computed using the point symmetries.

A \emph{quandle} is an algebraic system 
introduced in knot theory \cite{Joyce-1982-ClassifyingInvariantKnotsKnot},
that consists of a pair of a non-empty set $X$ 
and a map $s: X \to \Map(X, X)$ satisfying certain axioms.
These axioms align well with 
the local moves of knot diagrams, known as Reidemeister moves.
Recently,
quandles have played important roles in many branches
of mathematics.
For example, 
symmetric spaces can be viewed as quandles
via their point symmetries.
In other words,
a quandle can be considered as a generalization 
of an algebraic system 
that focuses specially on point symmetries in a symmetric space.
In particular,
quandles can be regarded as a discretization 
of symmetric spaces,
and there have been several recent studies from this view point
(for instance, see \cite{Furuki-2024-HomogeneousQuandlesAbelianInner,Kubo-2022-CommutativityConditionSubsetsQuandles,Ishihara-2016-FlatConnectedFiniteQuandles}).

A quandle structure gives rise to certain canonical groups,
which act on the quandle in a manner compatible with the point symmetries
$s_x := s(x): X \to X$.
The following group,
which plays a key role in this study, 
is defined by Joyce \cite[\S 5]{Joyce-1982-ClassifyingInvariantKnotsKnot}
as the transvection group.
\begin{dfn}[{\cite{Joyce-1982-ClassifyingInvariantKnotsKnot}},
  \cite{Hulpke-2016-ConnectedQuandlesTransitiveGroups}]
	The group defined by
	\begin{equation*}
		\Dis{X} = \langle s_x \circ s_y^{-1} \mid x,y \in X \rangle_\Grp
	\end{equation*}
	is called the \emph{displacement group} of a quandle $X$.		
\end{dfn}
In this paper,
we define the \emph{quandle Euler characteristic}
using the action of the displacement group as follows:
\begin{dfn}
  Let $X$ be a quandle.
  Then the \emph{quandle Euler characteristic} $\chi^{\Qdle}(X)$ is defined by 
  \begin{equation*}
    \chi^{\Qdle}(X) := \inf\{\# \Fix{g,X} \mid g \in \Dis{X}\},
  \end{equation*}
  where $\Fix{g,X}$ denotes the set of fixed points of the action of $g$ on $X$.
\end{dfn}
By definition, the quandle Euler characteristic satisfies 
$\chi^\Qdle(X) \in \ZZ_{\geq 0}$,
which aligns with the property of the topological Euler characteristics for
compact homogeneous spaces.
Furthermore,
the following theorem 
explains why we refer to the number defined above 
as the quandle Euler characteristic.
\begin{thm}[\cref{SymSpEuler}]
	For a compact connected Riemannian symmetric space,
	the quandle Euler characteristic is equal to 
	the topological Euler characteristic.
\end{thm}

In this paper, we prove several properties of 
quandle Euler characteristics,
which indicate that they 
share similarities with 
topological Euler characteristics.
In \cref{Example}, 
we determine the quandle Euler characteristics 
of some particular examples of quandles.

First examples we consider are generalized Alexander quandles.
A generalized Alexander quandle,
denoted by $\GAlex(G, \sigma)$, 
is defined as a group $G$ 
whose quandle structure is given by
a group automorphism $\sigma$ of $G$.
The class of generalized Alexander quandles
is important in quandle theory,
as any homogeneous quandle
can be realized as a quotient of some generalized Alexander quandle.
Note that any homogeneous manifold can be also
expressed as a quotient of a Lie group.
Therefore, 
in the theory of quandles,
generalized Alexander quandles can be viewed as 
counterparts of Lie groups.
The following theorem determines
the quandle Euler characteristics
of generalized Alexander quandles.
\begin{thm}[\cref{generalizedAlexander}]
	If $\sigma$ is a non-trivial group automorphism
  of a group $G$, 
	then the quandle Euler characteristic of 
  the generalized Alexander quandle $\GAlex(G, \sigma)$ is equal to $0$.
\end{thm}

The second examples are core quandles,
which are groups with quandle structures
derived from the map taking the inverse element.
This provides a natural quandle structure for a group
from the viewpoint of symmetric spaces.
In fact,  
a symmetric space structure of a Lie group 
is usually given by the core quandle structure.
Similar to the case of generalized Alexander quandles,
we determine the quandle Euler characteristics of core quandles,
which are analogous to the topological Euler characteristics
of non-trivial compact connected Lie groups that are equal to $0$.
\begin{thm}[\cref{core}]
  The quandle Euler characteristic of 
  a non-trivial core quandle is equal to $0$.
\end{thm}

The third examples consist of discrete subquandles 
within compact Riemannian symmetric spaces.
A typical example is 
the discrete $n$-sphere, denoted by $DS^n$, 
which is defined as 
a particular finite subquandle of the standard $n$-sphere $S^n$.
The following theorem suggests that 
the quandle Euler characteristics 
of certain finite objects
approximate the topological Euler characteristics 
of continuous objects.
\begin{thm}[\cref{DiscreteSphereEuler}]
	For any positive integer $n$, the quandle 
  Euler characteristic
  of the discrete $n$-sphere $DS^n$ satisfies
	$\chi^\Qdle(DS^n) = \chi^\Top(S^n)$, that is,
	\begin{equation*}
		\chi^{\Qdle}(DS^n) = 
		\begin{cases*}
			0 & if $n$ is odd,\\
			2 & if $n$ is even.
		\end{cases*}
	\end{equation*}
\end{thm}

We also consider the discrete torus $DT^n_u$,
which is a discrete subquandle of a flat torus $T^n = (S^1)^n$
parametrized by $u=(m_1, \dots, m_n) \in (\ZZ_{>0})^n$.
Recall that the topological Euler characteristic
of a flat torus is equal to $0$.
Similar to discrete spheres,
the quandle Euler characteristics of discrete tori
approximate the topological Euler characteristics of flat tori.
\begin{thm}[\cref{DiscreteTorusEuler}]
  For any positive integer $n$ 
  and any $u = (m_1, \dots, m_n) \in (\ZZ_{>0})^n$ with $m_i > 2$ 
  for all $i$,
  the quandle Euler characteristics of the discrete torus $DT^n_u$
  satisfies $\chi^\Qdle(DT^n_u) = \chi^\Top(T^n) = 0$.
\end{thm}

In \cref{Property},
we study some general properties of 
quandle Euler characteristics.
As a corollary of the K\"unneth theorem,
it is known 
that the topological Euler characteristic 
of the product of CW complexes
is equal to the product of 
the topological Euler characteristics of each components
(cf. \cite{Weintraub-2014-FundamentalsAlgebraicTopology}).
Note that the direct product 
of quandles can be defined naturally.
We prove a similar statement 
for the quandle Euler characteristic of the direct product:
\begin{thm}[\cref{directProductEuler}]
  The quandle Euler characteristic of the direct product of 
  quandles $(X_1,s^1)$ and $(X_2,s^2)$ satisfies
  $\chi^{\Qdle}(X_1 \times X_2) 
  = \chi^{\Qdle}(X_1) \cdot \chi^{\Qdle}(X_2)$.
\end{thm}
For two quandles,
one can naturally define a quandle structure on the disjoint union set,
which is called the \emph{interaction-free union} of quandles.
In the category of topological spaces,
the Euler characteristic of a disjoint union of spaces
is equal to the sum of the Euler characteristics 
of the components,
since the homology group of the disjoint union 
is isomorphic to the direct sum of 
the homology groups of the individual components.
We present an inequality concerning 
quandle Euler characteristics
that mirrors this property:
\begin{thm}[\cref{freeunionEuler}]
  Let $(X_1, s^1)$ and $(X_2, s^2)$ be quandles.
  Then, the quandle Euler characteristic of the interaction-free union 
  $X_1 \freeunion X_2$ satisfies
  \begin{equation*}
    \chi^{\Qdle}(X_1 \freeunion X_2) \leq \chi^{\Qdle}(X_1) + \chi^{\Qdle}(X_2).
  \end{equation*}
\end{thm}
Unlike the case of topological Euler characteristics,
the equality in the above theorem does not hold in general.
We will provide an example that does not achieve this equality 
(see \cref{counterExampleFreeunion}).


\section{Preliminary}\label{Preliminary}
In this section,
we review some notions of quandles
and basic facts on symmetric spaces.
In particular, 
the displacement groups play an important role.
First, let us recall the definition of quandles.

\begin{dfn}[\cite{Joyce-1982-ClassifyingInvariantKnotsKnot}]
	Let $X$ be a non-empty set and 
  let $\Map(X,X)$ be the set of all maps from $X$ to $X$. 
	For a map $s: X \to \Map(X,X)$,
  the pair $(X,s)$ is called a \emph{quandle} 
	if the following three conditions hold:
	\begin{enumerate}
		\item $s_x(x) = x$ for any $x \in X$,
		\item $s_x: X \to X$ is a bijection,
		\item $s_x \circ s_y = s_{s_x(y)} \circ s_x$ for any $x,y \in X$.
	\end{enumerate}
\end{dfn}

For quandles $(X,s^X)$ and $(Y, s^Y)$,
a map $f: X \to Y$ is called a \emph{quandle homomorphism}
if $f \circ s_x^X = s_{f(x)}^Y \circ f$ holds for any $x \in X$.
A bijective quandle homomorphism is 
called a \emph{quandle isomorphism}.

\begin{dfn}\label{qdleAut}
  The \emph{quandle automorphism group} $\Aut{X}$
  is the group consisting of all quandle 
  isomorphisms from $X$ to $X$.
  The quandle automorphism group is 
  a group under composition, 
  and acts on $X$.
  A quandle $X$ is called \emph{homogeneous} 
  if its automorphism group acts transitively on $X$.
\end{dfn}

Note that the maps $s_x: X \to X$ are quandle automorphisms.

\begin{dfn}\label{qdleInn}
  The subgroup of $\Aut{X}$ generated by 
  the set $\{s_x \mid x \in X\}$  
  is called the \emph{inner automorphism group} 
  and is denoted by $\Inn{X}$.
  A quandle $X$ is called \emph{connected} 
  if the inner automorphism group acts 
  transitively on $X$.
\end{dfn}

The following group is defined by Joyce 
\cite[\S 5]{Joyce-1982-ClassifyingInvariantKnotsKnot}
as the transvection group.

\begin{dfn}[{\cite{Joyce-1982-ClassifyingInvariantKnotsKnot}},
  \cite{Hulpke-2016-ConnectedQuandlesTransitiveGroups}]
	The subgroup of $\Inn{X}$ defined by
	\begin{equation*}
		\Dis{X} = \langle s_x \circ s_y^{-1} \mid x,y \in X \rangle_\Grp
	\end{equation*}
	is called the \emph{displacement group} of $X$.		
\end{dfn}
These groups and their actions are studied in detail 
in \cite{Hulpke-2016-ConnectedQuandlesTransitiveGroups}.
In particular, 
the actions of the displacement group 
and the inner automorphism group on a quandle 
have the same orbits.

According to the definition of 
a symmetric space given by 
Loos \cite{Loos-1969-SymmetricSpacesGeneralTheory},
its point symmetries provide a quandle structure.
For this reason,
the map $s_x: X \to X$ in the definition of a quandle
is called the \emph{point symmetry} 
at $x$ of $X$.
In the case of symmetric spaces,
Loos \cite{Loos-1969-SymmetricSpacesGeneralTheory}
gave the following properties 
about the displacement group.
These properties will be used 
in the latter argument.
\begin{prop}[{\cite[Chapter 2, Theorems 2.8 and 3.1]{Loos-1969-SymmetricSpacesGeneralTheory}}]
  \label{SymSpProp}
  Let $X$ be a symmetric space,
  and suppose that $X$ is connected as a topological space.
  Then the followings hold:
  \begin{enumerate}
    \item The displacement group $\Dis{X}$ is a connected Lie group.
    \item $X$ is connected as a quandle,
    and in particular $X$ is a homogeneous space of $\Dis{X}$.
  \end{enumerate}
\end{prop}

At the end of this section,
we see the $n$-dimensional unit sphere
as a typical example of symmetric spaces.
\begin{ex}\label{ex_sphere}
  The $n$-dimensional unit sphere $S^n$
  is a compact connected Riemannian symmetric space,
  and hence the point symmetries give rise to a quandle structure on $S^n$.
  The point symmetry $s_p: S^n \to S^n$ at a point $p \in S^n$ is defined by 
  \begin{equation*}
    s_p(x) = 2\langle x, p\rangle p - x,
  \end{equation*}
  where $\langle \cdot, \cdot \rangle$ is the standard inner product of $\RR^{n+1}$.
  In this case, the inner automorphism group $\Inn{S^n}$ 
  and the displacement group $\Dis{S^n}$ are
  \begin{align*}
    \Inn{S^n} &= \begin{cases*}
      \mathrm{O}(n+1) & if $n$ is odd,\\
      \mathrm{SO}(n+1) & if $n$ is even,
    \end{cases*}\\
    \Dis{S^n} &= \mathrm{SO}(n+1).
  \end{align*}
  Then, the quandle $S^n$ is connected,
  since the orthogonal group $\mathrm{O}(n+1)$ and 
  the special orthogonal group $\mathrm{SO}(n+1)$ 
  acts transitively on $S^n$.
  In particular, 
  the displacement group is a connected Lie group
  for any dimension $n$.
\end{ex}


\section{The quandle Euler characteristics}

In this section,
we define the quandle Euler characteristics
by using the displacement group.
We also show that this notion
is a natural generalization 
of the topological Euler characteristics
of compact connected symmetric spaces. 

When a group $G$ acts on a set $X$,
we denote the set of fixed points of this action by $\Fix{G,X}$.
Similarly,
we denote the set of fixed points of 
the action of an element $g \in G$ by $\Fix{g,X}$.
The following is the definition 
of the quandle Euler characteristics.

\begin{dfn}
	Let $X$ be a quandle.
	Then the \emph{quandle Euler characteristic} $\chi^{\Qdle}(X)$ is defined by 
	\begin{equation*}
		\chi^{\Qdle}(X) := \inf\{\# \Fix{g,X} \mid g \in \Dis{X}\}.
	\end{equation*}
\end{dfn}

In order to calculate the quandle Euler characteristic
of a given quandle $X$,
one needs to know the displacement group $\Dis{X}$.
Since the displacement group of a trivial quandle is 
the trivial group, we have the following result.

\begin{ex}
  The Euler characteristic of a trivial quandle 
  is equal to its cardinality.
\end{ex}


The definition of the quandle Euler characteristic
involves a group action.
This concept is inspired by a classical result 
from Hopf and Samelson 
\cite{Hopf-1941-SatzüberWirkungsräumeGeschlossener}.
Let $G$ be a compact connected Lie group
and $T$ a maximal torus of $G$.
Then,
there exists an element $g_0 \in T$
such that the topological closure in $G$ 
of the group generated by $g_0$ is equal to $T$.
Such an element $g_0 \in T$ is 
called a generator of $T$.
Hopf and Samelson provided a formula 
for the topological Euler characteristic of
a compact connected homogeneous space 
in terms of a maximal torus and its generator.
\begin{prop}[\cite{Hopf-1941-SatzüberWirkungsräumeGeschlossener}, 
  see also \cite{Püttmann-2002-HomogeneityRankAtomsActions}]\label{HS}
  Let $M$ be a homogeneous space 
  of a compact connected Lie group $G$.
  Let us consider
  a maximal torus $T$ of $G$,
  and let $g_0 \in T$ be a generator of $T$.
  Then, 
  the topological Euler characteristic $\chi^\Top(M)$ satisfies
  \begin{equation*}
    \chi^\Top(M) = \# \Fix{T, M} = \# \Fix{g_0, M}.
  \end{equation*}
\end{prop}

Using this proposition, 
one can show that the notion of quandle Euler characteristics
is a generalization of the notion of topological Euler characteristics of
compact connected Riemannian symmetric spaces.
Recall that a Riemannian symmetric space is a quandle by the point symmetries,
and each point symmetry is an isometry.

\begin{thm}\label{SymSpEuler}
	For a compact connected Riemannian symmetric space,
	the quandle Euler characteristic is equal to 
	the topological Euler characteristic.
\end{thm}

\begin{proof}
  Let $X$ be a compact connected Riemannian symmetric space.
  By \cref{SymSpProp},
  the displacement group $G := \Dis{X}$ is a connected Lie group,
  and acts transitively on $X$.
  Since each point symmetry is an isometry,
  the displacement group $G := \Dis{X}$ is 
  a subgroup of the isometry group $\mathrm{Isom}(X)$ of $X$.
  Since $X$ is a compact Riemannian symmetric space,
  the isometry group  is compact
  \cite[\S 5]{Myers-1939-GroupIsometriesRiemannianManifold},
  and in particular, $G$ is compact.
  Let $T$ be a maximal torus in $G$ and let $g_0$ be a generator of $T$. 
  Since 
  $\chi^\Top(X) = \# \Fix{g_0, X}$ by Proposition \ref{HS},
  we have
  \begin{equation*}
    \chi^\Top(X) \ge \inf_{g \in G} \Fix{g, X} = \chi^\Qdle(X).
  \end{equation*}
  
  It remains to prove 
  the converse inequality.
  Let us take any $g \in G$, and we will show
  \begin{equation*}
    \chi^{\Top}(X) \leq \# \Fix{g, X}.
  \end{equation*}
  Since $G$ is a compact Lie group,
  all maximal tori are conjugate.
  Hence there exists $h \in G$ such that
  $h g h^{-1} \in T$.
  Therefore, it satisfies
  \begin{equation*}
    \#\Fix{g, X} = \# \Fix{hgh^{-1}, X} \geq \# \Fix{T, X} = \chi^{\Top}(X),
  \end{equation*}
  which completes the proof.
\end{proof}

\section{Examples of the quandle Euler characteristics}\label{Example}

In this section, 
we calculate the quandle Euler characteristics 
of specific examples of quandles.

\subsection{The generalized Alexander quandles}

As the first example, 
we consider the generalized Alexander quandles.
For a group $G$ and a group automorphism $\sigma \in \Aut{G}$,
the \emph{generalized Alexander quandle} $\GAlex(G, \sigma)$
is a quandle $(G, s)$, 
where the map $s: G \to \Map(G,G)$ is
defined by $s_h(g) = h \sigma(h^{-1} g)$ for $g,h \in G$.
The inverse of a point symmetry $s_h$ 
is given by $s_h^{-1}(g) = h \sigma^{-1}(h^{-1} g)$.
Note that
a generalized Alexander quandle $\GAlex(G, \sigma)$
is trivial if and only if 
the automorphism $\sigma$ is trivial.
The next theorem 
determines the quandle Euler characteristics
of generalized Alexander quandles.

\begin{thm}\label{generalizedAlexander}
	If $\sigma$ is a non-trivial group automorphism
  of a group $G$, 
	then the quandle Euler characteristic of 
  the generalized Alexander quandle $\GAlex(G, \sigma)$ is equal to 
	$0$.
\end{thm}
\begin{proof}
  Since $\sigma$ is non-trivial,
  there exists $g \in G$ such that $\sigma(g) \neq g$.
  Consider $s_g \circ s_1^{-1} \in \Dis{\GAlex(G, \sigma)}$
  for $g \in G$,
  where $1$ is the identity of $G$.
  It is enough to show that 
  $s_g \circ s_1$ has no fixed points.
  For any $x \in G$, one has
  \begin{align*}
    s_g \circ s_1^{-1}(x) 
    = s_g(\sigma^{-1}(x))
    = g \sigma(g^{-1} \sigma^{-1}(x))
    = g \sigma(g)^{-1} x.
  \end{align*}
  Since $\sigma(g) \neq g$,
  we have $g \sigma(g)^{-1} \neq 1$,
  and thus $s_g \circ s_1^{-1}(x) \neq x$.
  This yields that
  $s_g \circ s_1^{-1}$ has no fixed points,
  which completes the proof.
\end{proof}

A generalized Alexander quandle is called 
an \emph{Alexander quandle}
(or an \emph{affine quandle})
if it is defined by an abelian group $G$.
Recall that the dihedral quandle $R_n$
is a typical example of Alexander quandles.

\begin{ex}\label{dihedralQdle}
  The \emph{dihedral quandle} $R_n$ 
  is a quandle isomorphic to 
  $\GAlex(\ZZ/ n\ZZ, \sigma)$,
  where $\sigma: \ZZ/ n\ZZ \to \ZZ/ n\ZZ$ is 
  the group automorphism defined by $\sigma(a) = -a$.
  The dihedral quandles can be realized as 
  discrete subquandles of the circle $S^1$
  (see \cref{R5}).
  Note that $R_n$ is non-trivial if $n>2$. 
  Hence, 
  for any integer $n > 2$, 
  it follows from \cref{generalizedAlexander} that
  \[
    \chi^{\Qdle}(R_n) = \chi^\Top(S^1) = 0.
  \]
\end{ex}

\begin{figure}[h]
  \centering
  \includegraphics[keepaspectratio, scale=0.4]{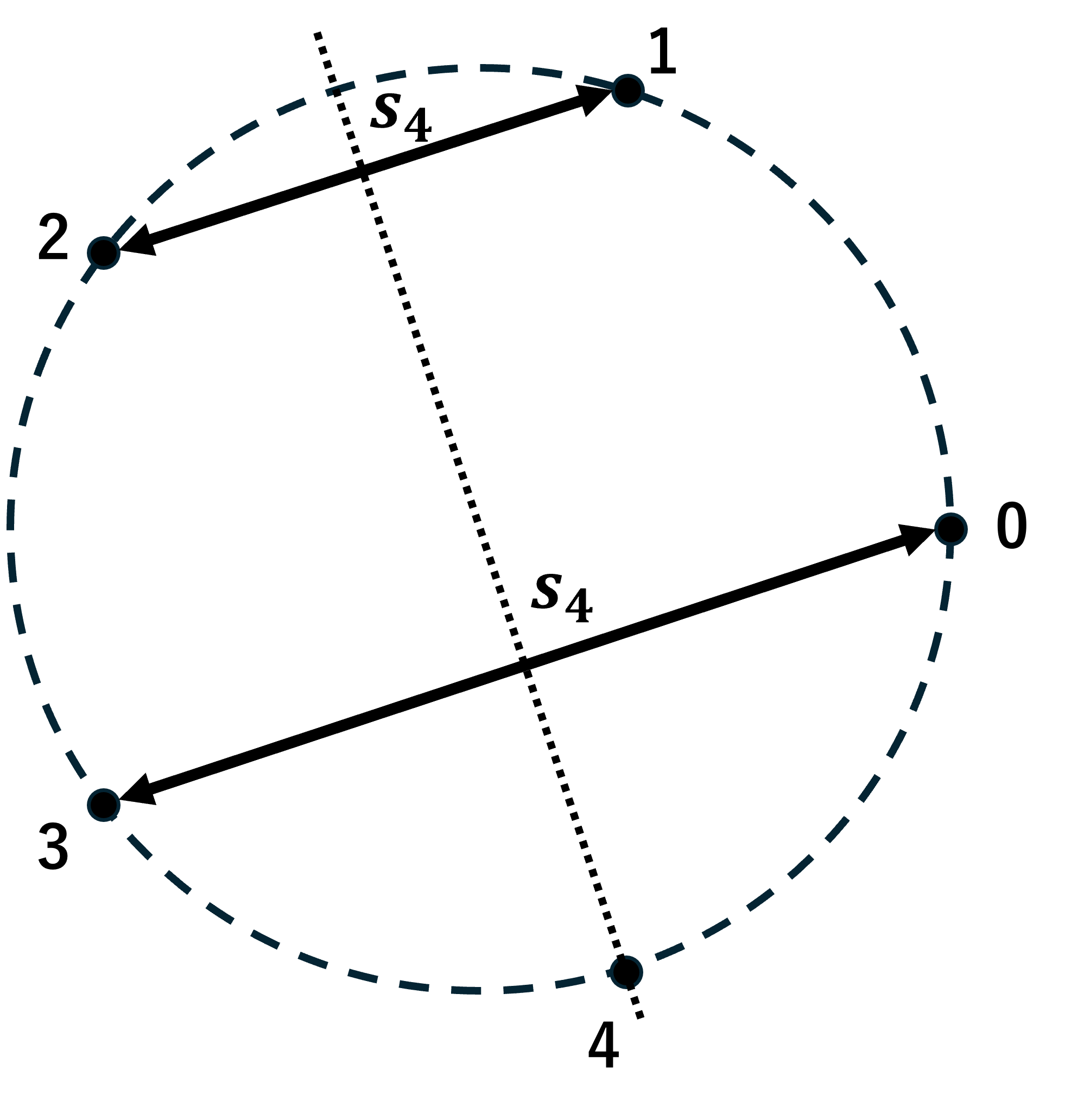}
  \caption{The dihedral quandle $R_5$ realized as a discrete subquandle of $S^1$.}
  \label{R5}
\end{figure}

\subsection{The core quandles}

In this subsection, 
we consider the core quandles.
For a group $G$,
the \emph{core quandle} $\Core(G)$ is a quandle $(G, s)$, 
where the map $s: G \to \Map(G,G)$ is
defined by $s_h(g) = hg^{-1}h$ for $g,h \in G$.
In particular, each point symmetry is an involution.
Note that the core quandle $\Core(G)$ is trivial if and only if
the order of any element in $G$ is less than or equal to $2$;
in other words,
the group $G$ is isomorphic to a direct product 
of some copies of $\ZZ/2\ZZ$.
Recall that 
any Lie group has a symmetric space structure
given by the core quandle structure.
The next theorem is an analogy of the fact that
the Euler characteristic of a non-trivial compact connected Lie group is 0.

\begin{thm}\label{core}
  The quandle Euler characteristic of 
  a non-trivial core quandle is equal to $0$.
\end{thm}
\begin{proof}
  Let $G$ be a group.
  If $G$ is abelian,
  the quandle $\Core(G)$ is isomorphic to 
  the Alexander quandle $\GAlex(G, \sigma)$,
  where the group automorphism $\sigma: G \to G$
  is given by $\sigma(g) = g^{-1}$ for $g \in G$.
  Since the quandle is non-trivial,
  the quandle Euler characteristic is equal to $0$
  by \cref{generalizedAlexander}.

  Let us assume that $G$ is not abelian.
  Then, there exist $g, h \in G$
  such that $[g, h]:=ghg^{-1}h^{-1} \neq 1$,
  where $1$ denotes the identity element.
  Consider the elements
  $g_1:= s_{g^{-1}h^{-1}} \circ s_{1}^{-1},
  g_2:= s_{h} \circ s_{g^{-1}}^{-1} \in \Dis{\Core(G)}$.
  For any $x \in \Core(G)$, we have
  \begin{align*}
    & g_1(x) = s_{g^{-1}h^{-1}} \circ s_1^{-1}(x)
    = s_{g^{-1}h^{-1}}(x^{-1}) = g^{-1}h^{-1} x g^{-1}h^{-1},\\
    & g_2(x) = s_{h} \circ s_{g^{-1}}^{-1}(x)
    = s_h(g^{-1}x^{-1}g^{-1}) = h g x g h.
  \end{align*}
  Hence, the map $\varphi := g_1 \circ g_2 \in \Dis{X}$
  satisfies
  \begin{align*}
    \varphi(x) = g_1 \circ g_2(x) = g_1(h g x g h)
    = g^{-1}h^{-1}(h g x g h)g^{-1}h^{-1}
    = x [g, h].
  \end{align*}
  Since $[g, h] \neq 1$,
  we have 
  $\varphi(x) \neq x$.
  Therefore,
  the map $\varphi$ has no fixed points,
  which completes the proof.
\end{proof}

\subsection{Discrete spheres}
In this subsection, 
we determine the Euler characteristics of discrete spheres,
which are defined as particular finite subquandles in spheres.
Recall that the $n$-dimensional unit sphere $S^n$ is a Riemannian symmetric space.
See \cref{ex_sphere}.

\begin{dfn}
  Let $\{e_i\}$ be the standard basis of $\RR^{n+1}$.
  Then the subset $DS^n :=\{\pm e_1, \dots, \pm e_{n+1}\}$ is 
  called the \emph{discrete $n$-sphere}. 
\end{dfn}
As in \cite[Example 2.4]{Furuki-2024-HomogeneousQuandlesAbelianInner},
it is easy to show that
$DS^n$ is a subquandle in $S^n$.
In fact,
the point symmetries satisfy
\begin{align*}
  s_{e_i} = s_{-e_i}, \qquad
  s_{e_{i}}(\pm e_i) = \pm e_i, \qquad
  s_{e_{i}}(\pm e_j) = \mp e_j \quad (\text{for } i \neq j).
\end{align*}
Note that these maps are the restrictions of 
some orthogonal transformations of $\RR^{n+1}$.
Hence,
we obtain an injective orthogonal representation 
$\rho: \Inn{DS^n} \to O(n+1)$ given by
\begin{equation*}
  \rho(s_{e_i}) = \mathrm{diag}(\varepsilon_1, \dots, \varepsilon_{n+1}),
  \quad
  \varepsilon_j = \begin{cases*}
    1 & if $i = j$,\\
    -1 & if $i \neq j$,
  \end{cases*}
\end{equation*}
where $\mathrm{diag}(\varepsilon_1, \dots, \varepsilon_{n+1})$
denotes the diagonal matrix whose $(i,i)$-entry is $\varepsilon_i$.
\begin{lem}
  The image $\rho(\Dis{DS^n})$ satisfies
  \begin{equation}\label{DiscreteSphereDis}
    \rho(\Dis{DS^n}) = 
    \{A =\mathrm{diag}(\varepsilon_1, \dots, \varepsilon_{n+1}) \mid
    \varepsilon_i \in \{\pm 1\}, \det(A) = 1\}.
  \end{equation}
\end{lem}
\begin{proof}
  In this case, 
  it is easy to see that
  \begin{align*}
    \rho(\Inn{DS^n}) \subset \left\{
    \mathrm{diag}(\varepsilon_1, \dots, \varepsilon_{n+1})  
    \in O(n+1)
    \mid \varepsilon_i \in \{\pm 1\} 
    \right\}.
  \end{align*}
  Recall that $S= \{s_{e_i} \circ s_{e_j}^{-1}\}$
  is a generating set of $\Dis{DS^n}$.
  For $i \neq j$,
  one can obtain the matrix expression
  \begin{equation*}
    \rho(s_{e_i} \circ s_{e_j}^{-1})
    = \mathrm{diag}(\varepsilon_1, \dots, \varepsilon_{n+1})
    ,\qquad
    \text{where }
    \varepsilon_k = \begin{cases*}
      1 & $k \not\in \{i, j\}$,\\
      -1 & $k \in \{i, j\}$.
    \end{cases*}
  \end{equation*}
  Hence, 
  the determinant of any element in $\rho(S)$ is equal to $1$.
  Since $\rho(\Dis{DS^n})$ is generated by $\rho(S)$,
  this shows the inclusion $(\subset)$ in the assertion.
  In order to prove the inverse inclusion, 
  let us take $B = \mathrm{diag}(\varepsilon_1, \dots, \varepsilon_{n+1})$
  in the right-hand side of the assertion.
  Then the number of indices $i$ with $\varepsilon_i = 1$ is even.
  Therefore the matrix $B$ can be written as 
  a product of elements in $\rho(S)$.
  This completes the proof.
\end{proof}

In the following,
we identify $\Inn{DS^n}$ with $\rho(\Inn{DS^n})$,
and also $\Dis{DS^n}$ with $\rho(\Dis{DS^n})$.
The next theorem determines 
the quandle Euler characteristics of the discrete spheres.

\begin{thm}\label{DiscreteSphereEuler}
	For any positive integer $n$, it satisfies
	$\chi^\Qdle(DS^n) = \chi^\Top(S^n)$, that is,
	\begin{equation*}
		\chi^{\Qdle}(DS^n) = 
		\begin{cases*}
			0 & if $n$ is odd,\\
			2 & if $n$ is even.
		\end{cases*}
	\end{equation*}
\end{thm}

\begin{proof}
	Suppose that $n$ is odd.
	Then the displacement group $\Dis{DS^n}$ 
  contains the identity matrix $-I_{n+1}$.
	The matrix $-I_{n+1}$ has no fixed points in $DS^n$. 
	Therefore we have $\chi^\Qdle(DS^n) = 0$.

	Suppose that $n$ is even.
	Then 
	we have 
	\begin{equation*}
		A_0 = \left(\begin{matrix}
			1 & 0\\
			0 & -I_{n}
		\end{matrix}\right) \in \Dis{DS^n}.
	\end{equation*}
	It is clear that 
	$A_0$ fixes two elements $\pm e_1$ in $DS^n$,
	which shows $\chi^\Qdle(DS^n) \leq 2$.
	Recall that any element $g \in \Dis{DS^n}$
	can be expressed as 
	\begin{align*}
		g = \mathrm{diag}(\varepsilon_1, \dots , \varepsilon_{n+1}),
	\end{align*}
	where $\varepsilon_i \in \{\pm 1\}$ and $\det(g) = 1$.
	Since $n+1$ is odd and hence $\#\{i \mid \varepsilon_i = 1\}$ is odd,
  we have $\#\{i \mid \varepsilon_i = 1\} \geq 1$.
	Therefore,
	$g$ has at least two fixed points in $DS^n$.
  This shows $\chi^\Qdle(DS^n) \geq 2$,
	which completes the proof.
\end{proof}

\subsection{Quandles obtained from weighted graphs}
  In this subsection, 
  we study the quandle Euler characteristics 
  of the quandles obtained from weighted graphs,
  as introduced in \cite{Saito-2025-HomogeneousQuandlesAbelianInner}.
  See also \cite{Furuki-2024-HomogeneousQuandlesAbelianInner}.

  \begin{prop}[{\cite{Saito-2025-HomogeneousQuandlesAbelianInner}}]
    Let $A$ be an abelian group, $V$ a finite set
    and $d: V \times V \to A$ a map 
    with $d(v, v) = 0$ for all $v \in V$.
    Then the following map $s: V \times A \to \Map(V \times A)$
    is a quandle structure on $V \times A$:
    \begin{equation*}
      s_{(v, a)}(w, b) = (w, d(v, w) + b).
    \end{equation*}
  \end{prop}

  The constructed quandle is called \emph{the quandle obtained from $(V, A, d)$},
  and denoted by $V \times_d A$.
  Note that $(V, A, d)$ is called an \emph{$A$-weighted graph}, 
  whose edges are weighted by elements in $A$.
  More precisely,
  $(V, A, d)$ is an oriented graph 
  with the vertex set $V$, 
  the edge set $E = \{(v,w) \mid v,w \in V, d(v,w) \neq 0\}$,
  and the weight of the edge $(v,w) \in E$ is given by $d(v,w)$.
  We express a weighted-edge $(v, w) \in E$
  by an arrow from $v$ to $w$ that is labeled by $d(v,w)$.

  Let us denote
  an $n$-point set by $V = V_n = \{v_1, \dots, v_n\}$
  for a positive integer $n$.
  For an $A$-weighted graph $(V, A, d)$,
  let us define \emph{the adjacency matrix} $D$ with respect to $(V, A, d)$
  by $D := (d(v_i, v_j))_{i,j \in \{1, \dots, n\}}$,
  and denote the $i$-th row vector of $D$ by $d_i$.
  
  \begin{rmk}
    For any $a, b \in A$ and any $v \in V_n$,
    the point symmetries satisfy $s_{(v, a)} = s_{(v, b)}$.
    In the following, 
    we denote the point symmetry at $(v, a) \in V_n \times_d A$ by $s_v$.
    Note that $s_{v_i}$ is described by $d_i$,
    which is the $i$-th row vector of the adjacency matrix $D$.
    Hence,
    $s_{v_i} = s_{v_j}$ as inner automorphisms of $V_n \times_d A$
    if and only if $d_i = d_j$.
    We can regard $\Inn{V_n \times_d A}$ as a subgroup of $A^n$
    by the map $\iota: \Inn{V_n \times_d A} \to A^n$ 
    defined by $\iota(s_{v_i})=d_i$.
    In the following we identify $\Inn{V_n \times_d A}$ with 
    its image of $\iota$.
  \end{rmk}

  \begin{rmk}
    It was proved in \cite{Saito-2025-HomogeneousQuandlesAbelianInner}
    that the inner automorphism group $\Inn{V \times_d A}$ is abelian,
    and every finite homogeneous quandle with abelian inner automorphism group
    can be constructed in this way.
    An $A$-weighted graph $(V, A, d)$ is said to be \emph{homogeneous}
    if there exists a graph isomorphism $f:(V,E) \to (V,E)$
    such that $d(f(v), f(w)) = d(v,w)$ for any $v,w \in V$.
    If an $A$-weighted graph $(V, A, d)$ is homogeneous,
    then $V_n \times_d A$ is a homogeneous quandle
    \cite[Theorem A]{Saito-2025-HomogeneousQuandlesAbelianInner}.
  \end{rmk}

  \begin{ex}
    The discrete sphere $DS^n$ is 
    isomorphic to the quandle $V_{n+1} \times_d (\ZZ / 2\ZZ)$,
    where the map $d: V_{n+1} \times V_{n+1} \to (\ZZ / 2\ZZ)$
    is defined by
    \begin{equation*}
      d(v,w) = \begin{cases*}
        1 & if $v \neq w$,\\
        0 & if $v = w$.
      \end{cases*}
    \end{equation*}
    The corresponding $A$-weighted graph is a perfect graph with $n+1$ vertexes,
    all of whose edges are weighted by $1 \in \ZZ/2\ZZ$.
    In the case $n = 3$,
    the corresponding $A$-weighted graph
    is described in \cref{exampleGraph},
    and 
    the corresponding adjacency matrix $D$ is 
    \begin{equation*}
      D = \begin{pmatrix}
        0 & 1 & 1 & 1 \\
        1 & 0 & 1 & 1 \\
        1 & 1 & 0 & 1 \\
        1 & 1 & 1 & 0 \\
      \end{pmatrix}.
    \end{equation*}
  \end{ex}

  \begin{figure}[h]
    \centering
    \includegraphics[keepaspectratio, scale=0.6]{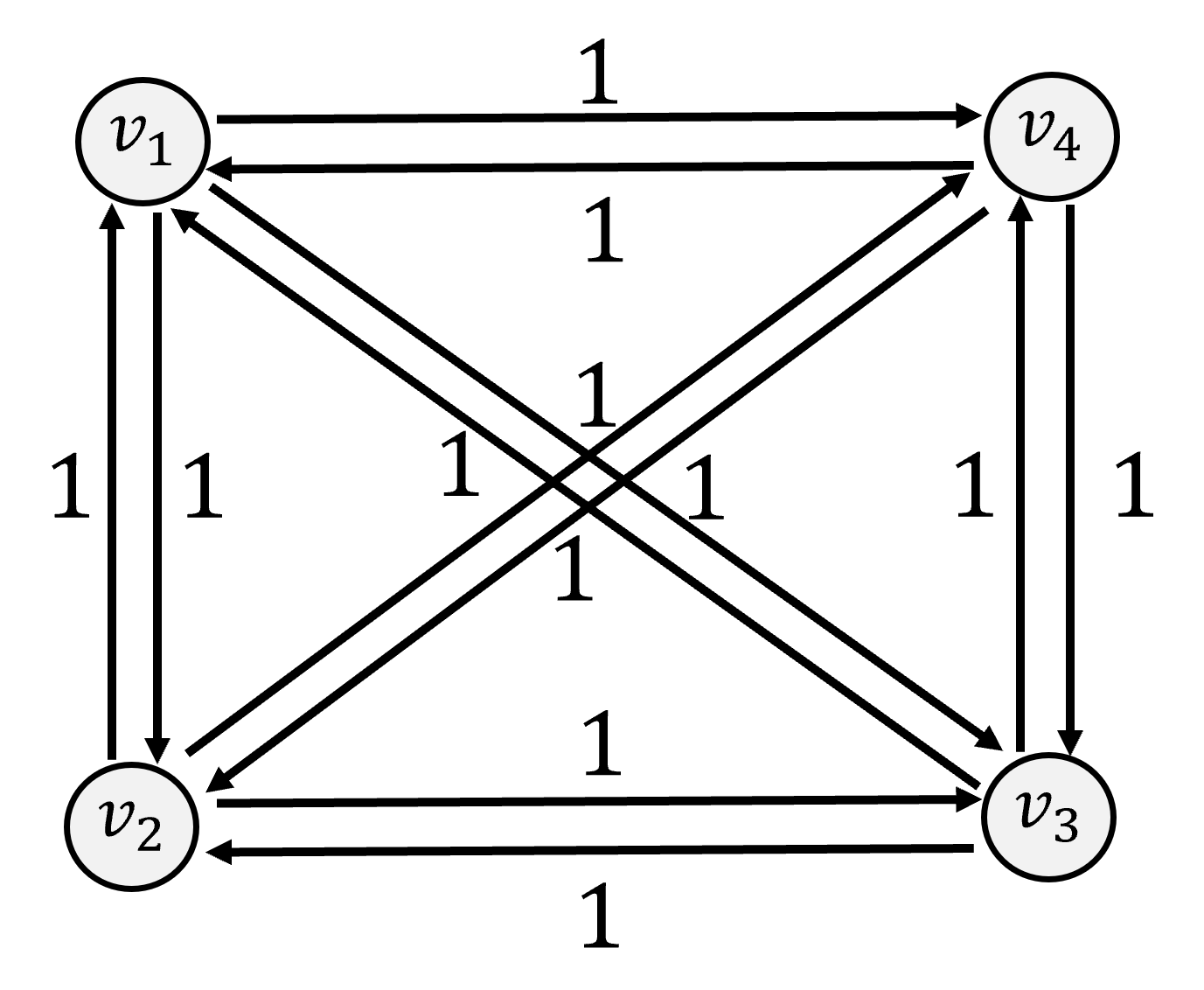}
    \caption{The $A$-weighted graph corresponding to 
    the discrete sphere $DS^3$}
    \label{exampleGraph}
  \end{figure}

  The next proposition is useful
  for calculating the displacement group 
  and the quandle Euler characteristic 
  of the quandle $V \times_d A$.
  \begin{prop}\label{graphEuler}
    Let $V_n \times_d A$ be the quandle obtained from 
    an $A$-weighted graph $(V, A, d)$,
    and $d_i$ be the $i$-th row vector of the adjacency matrix $D$.
    Then we have the following:
    \begin{description}
      \item[$(1)$] $\Dis{V_n \times_d A} =
      \left\{\sum_{i=0}^{n-1} k_i (d_i - d_{i+1}) \in A^n \mid k_i \in \ZZ \right\}$,

      \item[$(2)$] $\chi^\Qdle(V_n \times_d A) = \# A \cdot 
      \inf\left\{\#\{i \mid a_i = 0\} \mid (a_1,\dots, a_n) \in \Dis{V_n \times_d
       A}\right\}$.
    \end{description}
  \end{prop}
  \begin{proof}
    We show $(1)$.
    Recall that the displacement group $G:= \Dis{V_n \times_d A}$ is 
    generated by 
    \[
      \{s_{v_i} \circ s_{v_j}^{-1} \mid v_i, v_j \in V_n\}
      = \{d_i - d_j \mid i, j \in \{1, \dots n\}\}.
    \]
    Since $(d_i -d_j) = -(d_j - d_i)$ for any $i, j$,
    and $d_i - d_j = \sum_{k=i}^{j}(d_k - d_{k+1})$ for $j>i$,
    we can replace the generating set 
    to $\{d_i - d_{i+1} \mid i \in \{1, \dots, n-1\}\}$.
    Therefore we obtain the assertion.

    We show $(2)$.
    An element $a = (a_1, \dots, a_n) \in \Dis{V_n \times_d A}$
    acts on $V_n \times_d A$ by $a \cdot (v_i,b) = (v_i, b + a_i)$.
    Hence the element $(v_i, b) \in V_n \times_d A$ is fixed by $a$
    if and only if $a_i=0$.
    Note that 
    if there exists an element $(v_i, b_0) \in V_n \times_d A$ fixed by $a$,
    then for any $b \in A$, 
    every element $(v_i,b) \in V_n \times_d A$ is fixed by $a$.
    Hence
    the number $\# A \cdot \#\{i \mid a_i = 0\}$ 
    is equal to the number of elements fixed by $a$.
    This completes the proof.
  \end{proof}

  Now we calculate quandle Euler characteristics
  of a concrete example of $(V,A,d)$.
  This example will be used 
  to construct a quandle satisfying 
  particular inequality
  (see \cref{freeunionEuler} and \cref{counterExampleFreeunion}).
  \begin{prop}\label{cycleQuandle}
    Let $n$ be a positive integer with $n > 1$,
    and consider $V_n = \{v_1, \dots, v_n\}$. 
    Let us define a map $d: V_n \times V_n \to \ZZ/ 2\ZZ$ by 
    \begin{equation*}
      d(v_i, v_j) = \begin{cases*}
        1 & if $i-j = 1$ in $\ZZ/ n\ZZ$,\\
        0 & otherwise.
      \end{cases*}
    \end{equation*}
    Then $C_n := V \times_d (\ZZ/ 2\ZZ)$ is a non-trivial homogeneous quandle
    and the quandle Euler characteristic is given by
    \begin{equation*}
      \chi^\Qdle(C_n) = \begin{cases*}
        2 & if $n$ is odd,\\
        0 & if $n$ is even.
      \end{cases*}
    \end{equation*}
  \end{prop}

  \begin{figure}[h]
    \centering
    \includegraphics[keepaspectratio,scale=0.6]{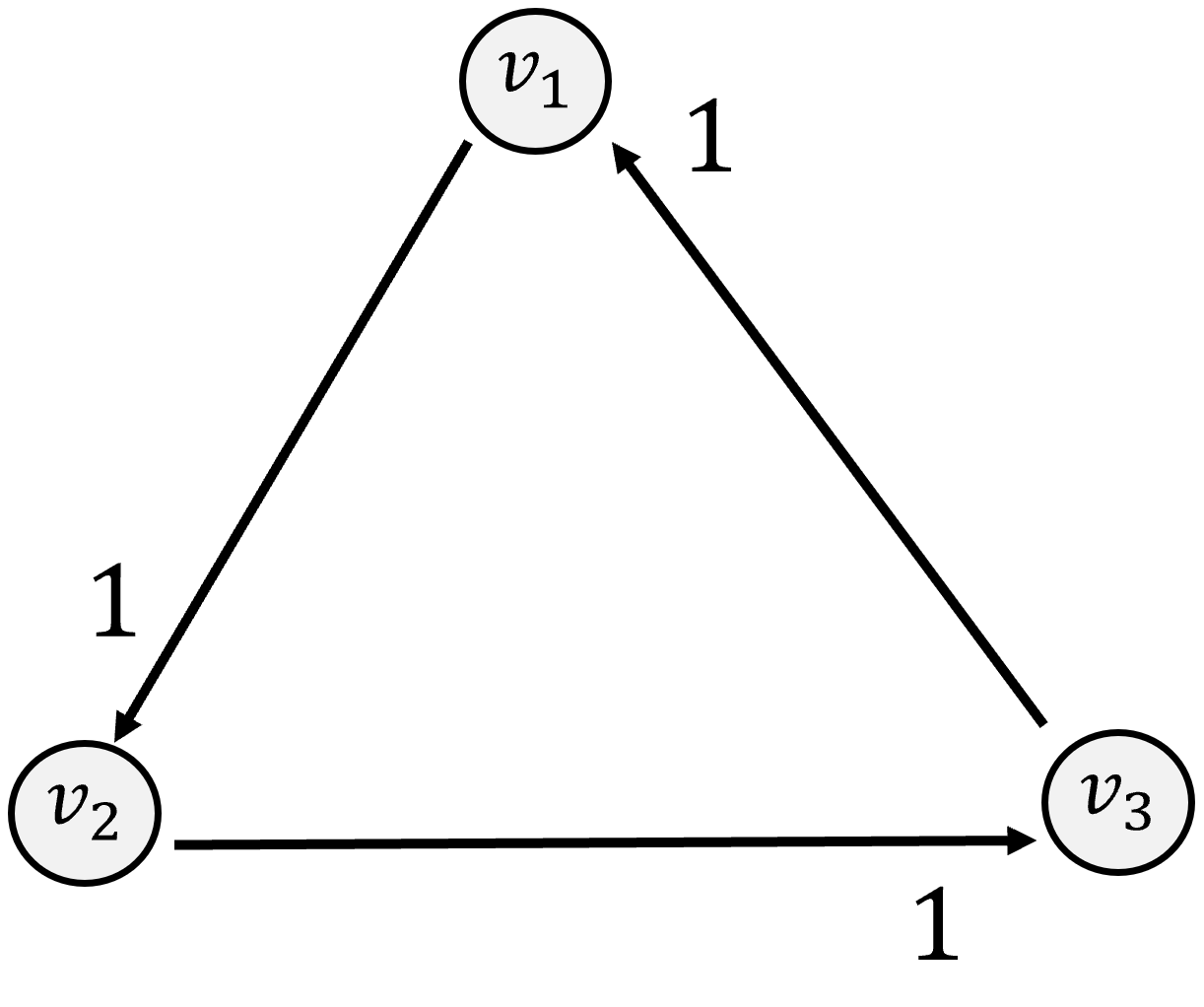}
    \caption{The $A$-weighted graph corresponding to 
    the quandle $C_3$.}
    \label{cycle3}
  \end{figure}

  \begin{proof}
    Since $d \neq 0$
    and the weighted graph $(V_n, \ZZ/2\ZZ, d)$ is homogeneous, 
    the quandle $C_n$ is non-trivial and homogeneous.
    In this case,
    the $i$-th row vector of the adjacency matrix is given by 
    \begin{equation*}
      d_i = e_{i-1}^T,
    \end{equation*}
    where $^T$ denotes the transpose.
    By \cref{graphEuler} $(1)$,
    we have
    \begin{equation*}
      \Dis{C_n} = 
      \left\{(a_1, \dots ,a_n) \in (\ZZ/2\ZZ)^n \mid a_k \in \ZZ/ 2\ZZ, \,\textstyle{\sum_{k=1}^{n}} a_k = 0\right\}.
    \end{equation*}
    If $n$ is even, then $(1,\dots, 1) \in \Dis{C_n}$ has no fixed points,
    and hence
    the quandle Euler characteristic is equal to $0$.
    If $n$ is odd, then $(1,\dots, 1, 0) \in \Dis{C_n}$
    and any element in $\Dis{C_n}$
    has at least one entry that is $0$.
    Hence the Euler characteristic is equal to $\#(\ZZ/2\ZZ) = 2$
    by \cref{graphEuler}~$(2)$.
    This completes the proof.
  \end{proof}

  The next proposition shows 
  that, given positive integer $n>1$,
  there exists a non-trivial quandle 
  whose Euler characteristic is equal to $n$.
  \begin{prop}
    Let $n > 1$ be a positive integer,
    and consider $V_2 = \{v_1, v_2\}$.
    Let us define a map $d: V_2 \times V_2 \to \ZZ/ n\ZZ$
    by 
    \begin{equation*}
      d(v_i, v_j) = \begin{cases*}
        1 & if $(i,j)=(1,2)$,\\
        0 & otherwise.
      \end{cases*}
    \end{equation*}
    Then $B_n := V \times_d \ZZ/ n\ZZ$ is a non-trivial quandle
    and the quandle Euler characteristic is given by
    \begin{equation*}
      \chi^\Qdle(B_n) = n.
    \end{equation*}
  \end{prop}

  \begin{figure}[h]
    \centering
    \includegraphics[keepaspectratio,scale=0.6]{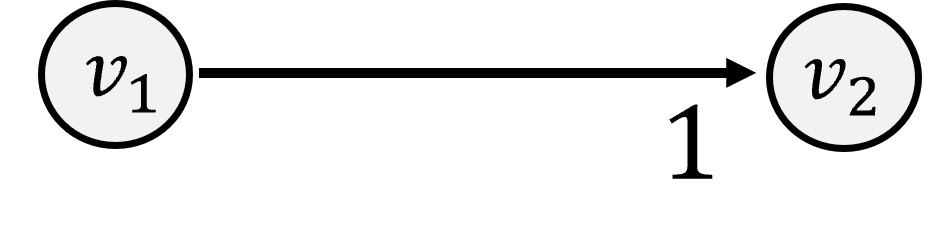}
    \caption{The $A$-weighted graph corresponding to 
    the quandle $B_n$.}
    \label{arc}
  \end{figure}

  \begin{proof}
    Since $d$ is not the zero-map, $B_n$ is non-trivial.
    In this case, 
    the displacement group satisfies
    \begin{equation*}
      \Dis{B_n} = \left\{(0, k) \in (\ZZ/ n\ZZ)^2 \mid k \in \ZZ/ n\ZZ \right\}
    \end{equation*}
    by \cref{graphEuler}.
    Since the first entry of any element in $W$ is $0$,
    the Euler characteristic of $B_n$ is equal to $\#(\ZZ/n\ZZ) = n$.
    This completes the proof.
  \end{proof}




\section{Properties of the quandle Euler characteristics}\label{Property}

In this section,
we prove several properties of 
the quandle Euler characteristics.
These properties can be regarded
as analogous to certain properties 
of topological Euler characteristics.

\subsection{The direct product of quandles}
Recall that the Euler characteristic of a product  
of CW complexes is equal to the direct product of 
the Euler characteristics of the individual components.
In this subsection,
we show that the same equality holds for the quandle Euler characteristic
of the product of quandles.

For quandles $(X_1, s^1)$ and $(X_2, s^2)$,
we can define a quandle structure $s$ on the direct product $X_1 \times X_2$
by setting $s_{(x_1,x_2)}:=s^1_{x_1} \times s^2_{x_2}$,
that is 
\begin{equation*}
	s_{(x_1,x_2)}(y_1,y_2) = (s^1_{x_1}(y_1), s^2_{x_2}(y_2)).
\end{equation*}
The quandle $(X_1 \times X_2, s)$ is called 
the \emph{direct product of quandles $(X_1,s^1)$ and $(X_2,s^2)$},
and for simplicity, we denote it as $X_1 \times X_2$.

\begin{lem}\label{prodGrp}
    Let $(X_1,s^1)$ and $(X_2,s^2)$ be quandles.
    Then
    \begin{description}
        \item[$(1)$]
        the group homomorphism
        $f: \Inn{X_1 \times X_2} \to \Inn{X_1} \times \Inn{X_2}$
        defined by $f(s_{(x_1,x_2)}) = (s^1_{x_1}, s^2_{x_2})$
        is injective,

        \item[$(2)$]
        $f(\Dis{X_1 \times X_2}) = \Dis{X_1} \times \Dis{X_2}$.
    \end{description}
\end{lem}

\begin{proof}
  We show $(1)$, that is $\ker(f) = \{1\}$.
  Let us take $g \in \ker(f)$. 
  Then there exist 
  $x_1, \dots, x_n \in X_1 \times X_2$
  and $\varepsilon_i \in \{\pm 1\}$ such that 
  $g=s_{x_1}^{\varepsilon_1} 
  \circ \cdots \circ s_{x_n}^{\varepsilon_n}$.
  Here, it follows
  $s_{x_i} = s_{x_i^1} \times s_{x_i^2}$ for each 
  $x_i = (x_i^1, x_i^2) \in X_1 \times X_2$
  by the definition of the point symmetries of the product quandle.
  By the definition of $f$, one has
  \begin{equation*}
    f(g) = ((s^1_{x_1^1})^{\varepsilon_1} \circ \cdots \circ (s^1_{x_n^1})^{\varepsilon_n}, 
    (s^2_{x_1^2})^{\varepsilon_1} \circ \cdots \circ (s^2_{x_n^2})^{\varepsilon_n})
    =: (g_1,g_2) \in \Inn{X_1} \times \Inn{X_2}.
  \end{equation*}
  Since $g \in \ker(f)$, 
  we have $g_i$ is trivial for $i \in \{1, 2\}$.
  On the other hand, it satisfies
  \begin{align*}
    g_1 \times g_2 
    &= \left((s^1_{x_1^1})^{\varepsilon_1} \circ \cdots \circ (s^1_{x_n^1})^{\varepsilon_n}\right) \times
    \left((s^2_{x_1^2})^{\varepsilon_1} \circ \cdots \circ (s^2_{x_n^2})^{\varepsilon_n}\right)\\
    &= \left((s^1_{x_1^1})^{\varepsilon_1} \times (s^2_{x_1^2})^{\varepsilon_1}\right)
    \circ \cdots \circ
    \left((s^1_{x_n^1})^{\varepsilon_n} \times (s^2_{x_n^2})^{\varepsilon_n}\right)\\
    &= s_{x_1}^{\varepsilon_1}
    \circ \cdots \circ
    s_{x_n}^{\varepsilon_n} = g.
  \end{align*}
  This yields that $g = 1$.
  Therefore the homomorphism $f$ is injective as desired.

  We show $(2)$.
  First, we demonstrate that 
  $f(\Dis{X_1 \times X_2}) \subset \Dis{X_1} \times \Dis{X_2}$.
  The left-hand side is generated by 
  \begin{equation*}
    D = \{(s^1_{x_1} \circ (s^1_{y_1})^{-1}, s^2_{x_2} \circ (s^2_{y_2})^{-1})
    \mid (x_1,x_2), (y_1, y_2) \in X_1 \times X_2\}.
  \end{equation*}
  Since $s_{x_k}^i \circ (s_{y_k}^i)^{-1} \in \Dis{X_i}$,
  the right-hand side contains $D$,
  which shows $(\subset)$.
  Next, we prove the converse inclusion $(\supset)$.
  The right-hand side $\Dis{X_1} \times \Dis{X_2}$ is generated by
  \begin{equation*}
    D' = \{(s^1_{x_1} \circ (s^1_{y_1})^{-1}, 1) \mid x_1, y_1 \in X_1\}
    \cup \{(1, s^2_{x_2} \circ (s^2_{y_2})^{-1}) \mid x_2, y_2 \in X_2\}.
  \end{equation*}
  The left-hand side contains $D'$.
  In fact, we have
  \begin{align*}
    (s^1_{x_1},s^2_{x_2}) \cdot (s^1_{y_1}, s^2_{x_2})^{-1}
    &= (s^1_{x_1} \circ (s^1_{y_1})^{-1}, 1),\\
    (s^1_{x_1}, s^2_{x_2}) \cdot (s^1_{x_1}, s^2_{y_2})^{-1}
    &= (1, s^2_{x_2} \circ (s^2_{y_2})^{-1}).
  \end{align*}
  Since both belong to the left-hand side,
  it follows $(\supset)$,
  which completes the proof.
\end{proof}

By this lemma,
the displacement group of the direct product quandle is 
isomorphic to the product group of the displacement groups of the individual components.
Using this property, we obtain the following result.

\begin{thm}\label{directProductEuler}
  The quandle Euler characteristic of the direct product of 
  quandles $(X_1,s^1)$ and $(X_2,s^2)$ satisfies
  $\chi^{\Qdle}(X_1 \times X_2) 
  = \chi^{\Qdle}(X_1) \cdot \chi^{\Qdle}(X_2)$.
\end{thm}

\begin{proof}
  We identify the group $\Dis{X_1 \times X_2}$ 
  with $\Dis{X_1} \times \Dis{X_2}$ by \cref{prodGrp}.
  Hence any element $g \in \Dis{X_1}$ can be written
  as $g = (g_1, g_2)$ with $g_1 \in \Dis{X_1}$ and $g_2 \in \Dis{X_2}$.
  Since the action is given by $(g_1, g_2). (x_1, x_2) = (g_1(x_1), g_2(x_2))$,
  it follows that the set of fixed points satisfies
  \begin{equation*}
    \Fix{(g_1, g_2), X_1 \times X_2} = \Fix{g_1, X_1} \times \Fix{g_2, X_2}.
  \end{equation*}
  Therefore, $(g_1, g_2)$ attains the infimum of $\# \Fix{(g_1, g_2), X_1 \times X_2}$
  if and only if both of $g_1$ and $g_2$ attain the infimums of $\# \Fix{g_1, X_1}$
  and $\# \Fix{g_2, X_2}$, respectively.
  This completes the proof of the desired equality.
\end{proof}

As an application of \cref{directProductEuler},
one can calculate the quandle Euler characteristics of the discrete tori.
As seen in \cref{dihedralQdle},
the dihedral quandles can be regarded as 
discrete subquandles of the circle $S^1$.
Thus,
we define the discrete tori as follows:
\begin{dfn}\label{DiscreteTorusDef}
  For a positive integer $n$
  and an integer vector 
  $u = (m_1, \dots, m_n) \in (\ZZ_{>0})^n$,
  the \emph{discrete torus} $DT^n_u$ is defined by 
  \[
    DT^n_u := \prod_{k=1}^n R_{m_k}.
  \]
\end{dfn}
According to the classification given in \cite{Ishihara-2016-FlatConnectedFiniteQuandles},
a flat connected finite quandle is a discrete torus.
Recall that a compact connected flat Riemannian symmetric space 
is just a flat torus $T^n = (S^1)^n$,
and its topological Euler characteristic is equal to $0$.
The Euler characteristic of a discrete torus
has the same property:
\begin{cor}\label{DiscreteTorusEuler}
  For any positive integer $n$ 
  and any $u = (m_1, \dots, m_n) \in (\ZZ_{>0})^n$ with $m_i > 2$,
  the quandle Euler characteristics of the discrete torus $DT^n_u$
  satisfies $\chi^\Qdle(DT^n_u) = \chi^\Top(T^n) = 0$.
\end{cor}
\begin{proof}
  Since $m_k>2$, any component $R_{m_k}$ is non-trivial.
  As seen in \cref{dihedralQdle},
  the quandle Euler characteristics of $R_{m_k}$ is equal to $0$.
  By applying \cref{directProductEuler},
  we have the quandle Euler characteristics of discrete torus
  is equal to $0$, which completes of proof.
\end{proof}

\subsection{The interaction-free union of quandles}
Recall that the Euler characteristic of a disjoint union 
of topological spaces is equal to the sum of 
the Euler characteristics of the individual components.
In this subsection,
we provide a comparable inequality for the Euler characteristics 
of the interaction-free union of quandles.

For two quandles $(X_1,s^1)$ and $(X_2,s^2)$,
we can define a quandle structure $s$ on the disjoint union $X_1 \sqcup X_2$
by setting 
\begin{equation*}
	s_x(y)=
	\begin{cases*}
		y & if $\{x,y\} \not\subset X_0 , X_1$,\\
		s_x^i(y) & if $\{x,y\} \subset X_i$.
	\end{cases*}
\end{equation*}
The quandle $(X_1 \sqcup X_2, s)$ is called the
\emph{interaction-free union of quandles $(X_1,s^1)$ and $(X_2,s^2)$},
and 
for simplicity,
we denote it by $X_1 \freeunion X_2$.
Note that the natural inclusion map 
$\iota_i: X_i \to X_1 \freeunion X_2$
is an injective quandle homomorphism
and induces a group homomorphism 
$\iota_i: \Inn{X_i} \to \Inn{X_1 \freeunion X_2}$.
\begin{lem}\label{unionGrp}
  Let $(X_1,s^1)$ and $(X_2,s^2)$ be quandles.
  Let us define the map 
  $\iota: \Inn{X_1} \times \Inn{X_2} \to \Inn{X_1 \freeunion X_2}$ by
  $\iota(g_1, g_2) = \iota_1(g_1)\iota_2(g_2)$.
  Then,
  \begin{enumerate}
    \item[$(1)$]\label{unionInn}
    the map $\iota: \Inn{X_1} \times \Inn{X_2} \to \Inn{X_1 \freeunion X_2}$
    is a group isomorphism,

    \item[$(2)$]\label{unionDis}
    $\iota(\Dis{X_1} \times \Dis{X_2}) \subset \Dis{X_1 \freeunion X_2}$.
  \end{enumerate}
\end{lem}

\begin{proof}
  We show $(1)$,
  that is, $\iota$ is a group homomorphism and is a bijection.
  First, we show that $\iota$ is a group homomorphism.
  For $x_1 \in X_1$ and $x_2 \in X_2$,
  the maps $s_{x_1}$ and $s_{x_2}$ 
  are commutative in $\Inn{X_1 \freeunion X_2}$
  by the definition of the interaction-free union.
  Thus,
  an element in $\iota_1(X_1)$ and an element in $\iota_2(X_2)$ are commutative.
  Hence for $(g_1, g_2), (h_1, h_2) \in \Inn{X_1} \times \Inn{X_2}$,
  we have
  \begin{align*}
    \iota((g_1, g_2) (h_1, h_2)) 
    &= \iota(g_1 h_1, g_2 h_2)\\
    &= \iota_1(g_1 h_1) \iota_2(g_2 h_2)\\
    &= \iota_1(g_1) \iota_1(h_1) \iota_2(g_2) \iota_2(h_2)\\
    &= \iota_1(g_1) \iota_2(g_2) \iota_1(h_1) \iota_2(h_2)\\
    &= \iota(g_1,g_2) \iota(h_1, h_2).
  \end{align*}
  Therefore
  the map $\iota$ is a group homomorphism.
  Next, we show that $\iota$ is surjective.
  To prove this,
  let us consider
  \[
    S := \{(s^1_{x_1},1) \mid x_1 \in X_1\} \sqcup \{(1,s^2_{x_2}) \mid x_2 \in X_2\},
  \]
  which is a generating set of $\Inn{X_1} \times \Inn{X_2}$.
  Note that $\iota(s_{x_1}^1,1) = s_{x_1}$ and $\iota(1,s_{x_2}^2) = s_{x_2}$.
  Then the image $\iota(S)$ satisfies
  \begin{equation*}
    \iota(S) = \{s_x \in \Inn{X_1 \freeunion X_2} \mid x \in X_1 \sqcup X_2\},
  \end{equation*}
  which is a generating set 
  of $\Inn{X_1 \freeunion X_2}$.
  Thus, the group homomorphism $\iota$ is surjective.
  Lastly, we show that $\iota$ is injective,
  that is $\ker \iota = \{1\}$.
  Let us take $(g_1, g_2) \in \ker(\iota)$, 
  where $g_1 \in \Inn{X_1}$ and $g_2 \in \Inn{X_2}$. 
  Then the map $\iota (g_1, g_2)=\iota_1(g_1) \iota_2(g_2)$ 
  acts trivially on $X_1 \freeunion X_2$.
  In particular,
  the element $\iota_1(g_1) = \iota(g_1, g_2) \iota_2(g_2)^{-1}$ 
  acts trivially on $X_1$.
  Thus we conclude that $g_1$ is trivial,
  and similarly we conclude that $g_2$ is trivial.
  Therefore,
  we have $(g_1, g_2) = (1, 1)$ as desired.

  We show $(2)$.
  Let us consider \[
    S':= \{(s^1_{x_1} \circ (s^1_{y_1})^{-1},1) \mid x_1, y_1 \in X_1\} 
    \sqcup \{(1,s^2_{x_2}\circ (s^2_{y_2})^{-1}) \mid x_2, y_2 \in X_2\},
  \]
  which is a generating set of $\Dis{X_1} \times \Dis{X_2}$.
  Then the image $\iota(S')$ is given by 
  \[
    \iota(S') = 
      \{s_{x_1} \circ (s_{y_1})^{-1} \mid x_1, y_1 \in X_1\} 
    \sqcup \{s_{x_2}\circ (s_{y_2})^{-1} \mid x_2, y_2 \in X_2\}
    ,
  \]
  which is a subset of the group $\Dis{X_1 \freeunion X_2}$.
  Since the group $\iota(\Dis{X_1} \times \Dis{X_2})$
  is generated by $\iota(S')$,
  we have $(2)$, which completes the proof.
\end{proof}
By this lemma,
we can regard 
the product group of the displacement groups of the individual components
as a subgroup of the displacement group of the interaction-free union quandle.
Using this property,
we obtain the following result.

\begin{thm}\label{freeunionEuler}
  Let $(X_1, s^1)$ and $(X_2, s^2)$ be quandles.
  Then, the quandle Euler characteristic of the interaction-free union 
  $X_1 \freeunion X_2$ satisfies
  \begin{equation*}
    \chi^{\Qdle}(X_1 \freeunion X_2) \leq \chi^{\Qdle}(X_1) + \chi^{\Qdle}(X_2).
  \end{equation*}
\end{thm}
\begin{proof}
  We regard $\Dis{X_1} \times \Dis{X_2}$
  as a subgroup of 
  $\Dis{X_1 \freeunion X_2}$
  by the map $\iota$ in \cref{unionGrp}.
  For each $i \in \{1,2\}$,
  there exists $g_i \in \Dis{X_i}$ 
  such that 
  \begin{equation*}
      \#\Fix{g_i, X_i} = \chi^\Qdle(X_i).
  \end{equation*}
  Then $g := g_1g_2$ is regarded as an element in $\Dis{X_1 \freeunion X_2}$.
  Since $g_i$ acts trivially on the other component,
  it follows that the set of fixed points satisfies
  \begin{equation*}
      \Fix{g, X_1 \freeunion X_2} = \Fix{g_1, X_1} \sqcup \Fix{g_2, X_2}.
  \end{equation*}
  This concludes that
  \begin{equation*}
      \chi^\Qdle(X_1 \freeunion X_2) 
      \leq \# \Fix{g, X_1 \freeunion X_2} 
      = \# \Fix{g_1, X_1} + \# \Fix{g_2,X_2} 
      = \chi^\Qdle(X_1)+\chi^\Qdle(X_2),
  \end{equation*}
  which completes the proof.
\end{proof}

The following provides an example
that does not satisfy the equality in \cref{freeunionEuler}.
We will use a quandle obtained from a weighted graph.
Note that this quandle is homogeneous,
and therefore,
the equality in \cref{freeunionEuler}
does not hold in general,
even for homogeneous quandles.

\begin{ex}\label{counterExampleFreeunion}
  Let $C_3$ be the quandle in \cref{cycleQuandle}.
  Then we have
  \[
    \chi^\Qdle(C_3 \freeunion C_3) = 0
    < 4 = \chi^\Qdle(C_3) + \chi^\Qdle(C_3).
  \]
\end{ex}

\begin{proof}
  Since we proved $\chi^\Qdle(C_3) = 2$
  in \cref{cycleQuandle},
  we have only to show that
  the Euler characteristic of 
  $C_3 \freeunion C_3$ is equal to $0$.
  Recall that $C_3$ is obtained
  from the $\ZZ/ 2\ZZ$-weighted graph
  in \cref{cycle3}.
  We have that
  $C_3 \freeunion C_3$ coincides with 
  the quandle obtained from the 
  $\ZZ/2\ZZ$-weighted graph in \cref*{c3c3},
  which is the disjoint union of two copies of 
  the graph of $C_3$.
  With respect to the labeling in \cref{c3c3},
  the corresponding adjacency matrix $D$ is given by
  \[
  D = \begin{pmatrix}
    0 & 1 & 0 & 0 & 0 & 0\\
    0 & 0 & 1 & 0 & 0 & 0\\
    1 & 0 & 0 & 0 & 0 & 0\\
    0 & 0 & 0 & 0 & 1 & 0\\
    0 & 0 & 0 & 0 & 0 & 1\\
    0 & 0 & 0 & 1 & 0 & 0
  \end{pmatrix}.
  \]
  It follows from \cref*{graphEuler} that
  \[
      (1,1,1,1,1,1)=
      (d_1 - d_2) + (d_3 - d_4) + (d_5 - d_6)
    \in \Dis{C_3 \freeunion C_3}.
  \]
  Since this element has no fixed points,
  we have $\chi^\Qdle(C_3 \freeunion C_3) = 0$
  by \cref{graphEuler} $(2)$.
\end{proof}
\begin{figure}[h]
  \centering
  \includegraphics[keepaspectratio,scale=0.6]{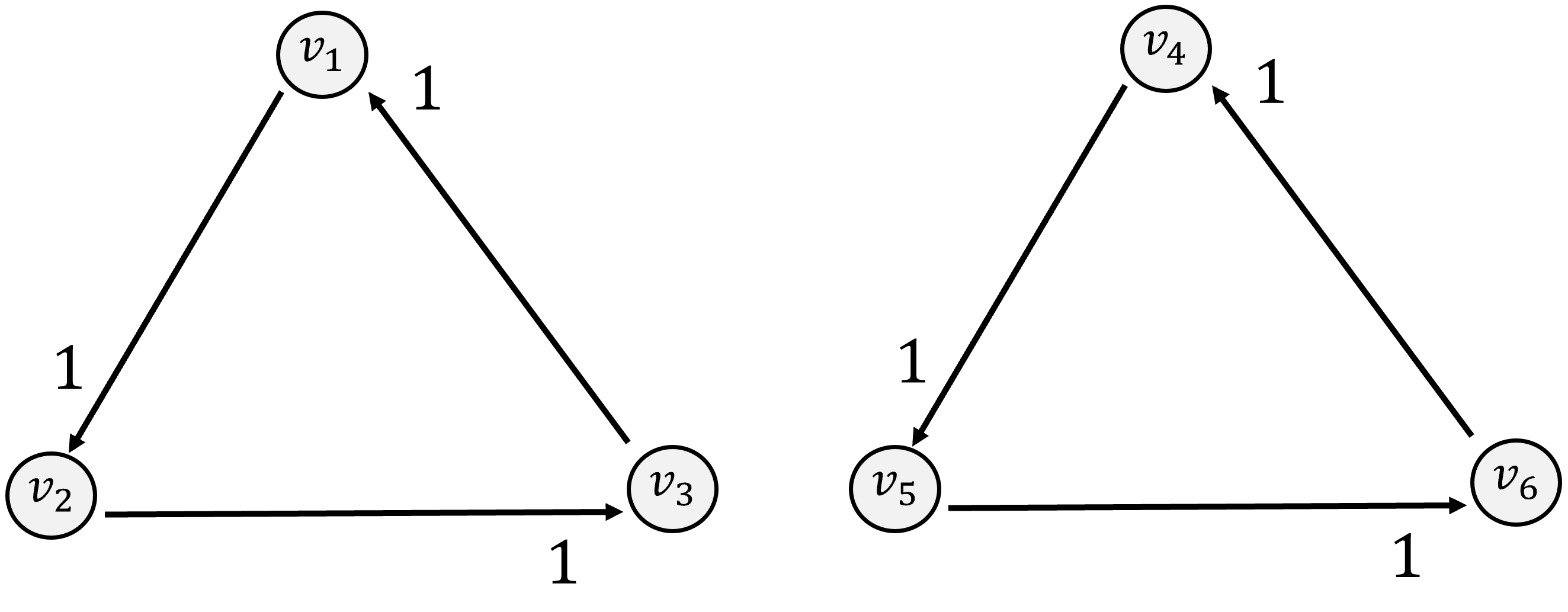}
  \caption{The $A$-weighted graph corresponding to 
  the quandle $C_3 \freeunion C_3$.}
  \label{c3c3}
\end{figure}

\section*{Acknowledgement}
The authors would like to thank 
Hirotaka Akiyoshi,
Katsunori Arai,
Seiichi Kamada,
Akira Kubo,
Fumika Mizoguchi,
Takayuki Okuda,
Makoto Sakuma,
and Yuta Taniguchi
for helpful comments
and useful discussions.
The first author 
was supported by JST SPRING, Grant Number JPMJSP2139.
The second author was supported by 
JSPS KAKENHI Grant Numbers JP22H01124 
and JP24K21193.
The authors were partly supported by 
MEXT Promotion of Distinctive Joint Research
Center Program JPMXP00723833165.

\bibliographystyle{spmpsci}
\bibliography{qdle_euler}

\vspace{-3mm}

\address{
	(R. Kai) 
	Department of Mathematics, 
	Graduate School of Science, 
	Osaka Metropolitan University, 
	3-3-138, Sugimoto, 
	Sumiyoshi-ku, Osaka, 558-8585, Japan}
\email{sw23889b@st.omu.ac.jp}


\address{
	(H. Tamaru) 
	Department of Mathematics, 
	Graduate School of Science, 
	Osaka Metropolitan University, 
	3-3-138, Sugimoto, 
	Sumiyoshi-ku, Osaka, 558-8585, Japan}

\email{tamaru@omu.ac.jp}

\end{document}